\documentclass[12pt]{article}

\usepackage{amssymb}%
\usepackage{amsmath}%
\usepackage{amsthm}%
\usepackage{fullpage}
\usepackage{booktabs}
\usepackage{float}
\usepackage{xcolor}%

\allowdisplaybreaks%

{\theoremstyle{plain}%
 \newtheorem{theorem}{Theorem}
 
 \newtheorem{lemma}{Lemma}%
}
{\theoremstyle{remark}
\newtheorem{remark}{Remark}
}
{\theoremstyle{definition}

}
\newcommand{\Li}{\operatorname{Li}}
\newcommand{\Cl}{\operatorname{Cl}}

\usepackage[
  colorlinks=true,
  linkcolor={blue!55!black},
  citecolor={teal!70!black},
  urlcolor={blue!55!black},
  linktoc=all,
  bookmarksopen=true,
  bookmarksnumbered=true,
  breaklinks=true,
  pdfstartview=FitH,
  pdftitle={Functional Dilogarithm Identities in Quadratic Fields},
  pdfauthor={Cetin Hakimoglu-Brown}
]{hyperref}
\usepackage[nameinlink,capitalize]{cleveref}

\crefname{equation}{Eq.}{Eqs.}
\Crefname{equation}{Equation}{Equations}

\begin{document}

\begin{center}
{\Large Functional Dilogarithm Identities in Quadratic Fields}

 \  

Cetin Hakimoglu-Brown\\
mathemails@proton.me

 \ 

\end{center}

\begin{abstract}
We derive three- and six-term functional dilogarithm identities whose arguments lie in
$\mathbb{Q}(u,\sqrt{4-3u^2})$ and $\mathbb{Q}(u,\sqrt{u(4-3u)})$. Our approach is based on
an integral-to-${}_4F_3$ correspondence that converts families of cubic and sextic
integrals into hypergeometric identities, providing a systematic method for constructing
functional equations for the dilogarithm over quadratic fields. We demonstrate the power
of this method by giving an analytic proof of the classical Loxton--Lewin
$2\cos(4\pi/9)$ identity, deriving a new family of dilogarithm ladders of quartic base
lying in $\mathbb{Q}(\sqrt{33})$, and proving conjectural two-term identities of Bytsko. As
a further application, we obtain rapidly convergent ${}_4F_3$ series for
$\mathrm{Cl}_2(\pi/3)$ and explicit relations connecting $\mathbb{Q}(\sqrt{13})$ and
$\mathbb{Q}(\sqrt{3})$. Finally, a PSLQ-based search over palindromic quartic units, we report  new ladder relations with
arguments built from $2\tan(\pi/8)\cos(\pi/5)$ and $\tan(3\pi/20)$, analogous to known trig identities by Watson, Loxton, and Gordon–McIntosh.
\end{abstract}


\vspace{0.1in}


\vspace{0.2in}
{\small\tableofcontents}
\newpage

\section{Preliminaries, Procedure, and Main Results}\label{sec:prelim}

 \noindent The \emph{dilogarithm function} is a central object of study in our work and is given by the $m = 2$ case of the 
 \emph{polylogarithm function} such that $\operatorname{Li}_{m}(x) := \sum_{n=1}^{\infty} \frac{x^{n}}{n^m}$ for  algebraic   $x$ 
 such that $|x| \leq 1$. 
\\
 
 \noindent Define: 
\begin{equation}\label{RogersLdefinition}
 \operatorname{L}(x)= \operatorname{Li}_{2}(x) + \frac{1}{2} \ln x \ln(1-x)  
\end{equation}
 to denote the Rogers $L$-function. We also have Roger's five-term identity:
\begin{equation}\label{rogers}
 \operatorname{L}(x) + \operatorname{L}(y) = \operatorname{L}(x y) + 
\operatorname{L}\left( \frac{x(1-y)}{1- x y} \right) + \operatorname{L}\left( \frac{y(1-x)}{1-xy} \right) 
\end{equation} 
In addition to the duplication formula: 
\begin{equation}\label{duplicationformula}
 2 \operatorname{Li}(x) + 2 \operatorname{Li}(-x) = \operatorname{Li}(x^{2}) 
\end{equation}
Euler's reflection identity: 
\begin{equation}\label{Eulerreflectionidentity}
 \operatorname{Li}_2(x) = -\operatorname{Li}_2(1 - x) + \frac{\pi^2}{6} - \ln(x) \ln(1 - x) 
\end{equation}
Landen's  identity: 
\begin{equation}\label{Landenformula}
 \operatorname{Li}_{2}(-x) = -\operatorname{Li}_{2}\left( \frac{x}{1+x} \right) - \frac{1}{2} \ln^2(1+x) 
\end{equation}

\noindent And the four-term relation: 
\begin{equation}\label{wellknown4term}
 \operatorname{Li}_2(x) - \operatorname{Li}_2(-x) + 
 \operatorname{Li}_2\left(\frac{1 - x}{1 + x}\right) - 
 \operatorname{Li}_2\left(\frac{x-1}{1 + x}\right) = \frac{\pi^2}{4} + \ln(x) \ln\left(\frac{1 + x}{1 - x}\right) 
\end{equation}

\remark For brevity, we will typically make use of the $\operatorname{L}$ notation for ladders and functional equations, although for certain results switch to $\operatorname{Li_2}$.

\subsection{Main theorems} 
 The principal contribution of this paper is an
integral-to-${}_4F_3$ framework for constructing functional
dilogarithm identities over quadratic fields. Unlike previous
approaches based on repeated applications of the five-term relation \eqref{rogers},
our method begins with integral representations whose hypergeometric
evaluations lead directly to functional equations. We illustrate the
method by giving analytic proofs of classical identities and by
deriving new quartic-base ladders. This builds off of results by Abouzahra and Lewin, Gordon and McIntosh, and Kirilov. For a full review of background material on the dilogarithm, we again refer to the seminal monographs due to 
 Lewin concerning this function \cite{Lewin1958,Lewin1981,Lewin1991text}. 

We state our main results

\begin{theorem}\label{firstformatladder}
\begin{multline}
 \operatorname{L}\left( \frac{u^2 + \sqrt{4 - 3u^2} - 2}{-u(u + 1)} \right) 
 - 3 \operatorname{L}\left( \frac{u^2 + \sqrt{4 - 3u^2} - 2}{u(u + 1)} \right) 
 - 4 \operatorname{L}\left( \frac{1 - u}{u + 1} \right) \\ 
 + \operatorname{L}\left( \frac{u^2 - \sqrt{4 - 3u^2} - 2}{-u(u + 1)} \right) 
 - 3 \operatorname{L}\left( \frac{u^2 - \sqrt{4 - 3u^2} - 2}{u(u + 1)} \right) 
 + \frac{3}{2} \operatorname{L}\left( \left( \frac{1 - u}{u + 1} \right)^2 \right)\in \pi^2 \mathbb{Q} 
\label{eq:ladder2} 
\end{multline}
\end{theorem}

\noindent This has reflection viz. \eqref{duplicationformula} and symmetrical  properties e.g. $\left(\frac{u^2 - \sqrt{4 - 3u^2} - 2}{u(u + 1)}\right)^{-1}
= \left(\frac{u^2 + \sqrt{4 - 3u^2} - 2}{u(u - 1)}\right)$ that makes it suitable for ladders.

\begin{theorem}\label{previouslymaster}
For $|u|<1$, we have
\begin{multline}
- \operatorname{Li}_2\!\left( \frac{2u^2 - 3u + \sqrt{u(4-3u)}}{2(u-1)^2} \right)
- \operatorname{Li}_2\!\left( \frac{2u^2 - 3u - \sqrt{u(4-3u)}}{2(u-1)^2} \right) \\
= \frac{3}{2} \ln^2\!\left( 1 - \frac{u - \sqrt{(4 - 3u)u}}{2(u-1)} \right)
- \operatorname{Li}_2\!\left( \frac{u}{u - 1} \right).
\label{eq:ladder}
\end{multline}
\end{theorem}

\subsubsection{Key Results}

In \Cref{subsec:sqrt13cl2} we derive an unexpected relation between $\mathbb{Q}(\sqrt{13})$ , $\mathbb{Q}(\sqrt{3})$ , and   $\operatorname{Cl}_2\left(\frac{\pi}{3}\right)$ (Gieseking's constant) :

\begin{equation}
3\,\Im\operatorname{Li}_2(z)-\Im\operatorname{Li}_2(-z)
=\frac{11}{6}\,\operatorname{Cl}_2\!\left(\frac{\pi}{3}\right)
-\frac{\pi}{6}\,\cosh^{-1}\!\left(\frac{19}{6}\right),
\qquad
z=\frac{5-\sqrt{13}}{\sqrt{12}}\,e^{i\pi/6}.
\end{equation}

\noindent We  prove the dilogarithm ladder: 

$$ 3 \operatorname{Li}_{2}(r^6) + 3 \operatorname{Li}_{2}(r^4) 
 - 8 \operatorname{Li}_{2}(r^3) - 33 \operatorname{Li}_{2}(r^2) + 24 \operatorname{Li}_{2}(r) - \zeta(2) = 
 6 \log^2(r). $$ 

\noindent where \begin{equation*}
 r=\frac{1}{4} \left( -1 + \sqrt{33} - \sqrt{2 \left(9 - \sqrt{33} \right)} \right) 
\end{equation*}

\remark{The minimal polynomial of $r$ has all real roots, analogous to the trio of ladders found by Watson \cite{Watson1937}, Lewin-Loxton \cite{Lewin1982,Loxton}, and Gordon-McIntosh \cite{GordonMcIntosh1997}, whose ladder-producing cubic minimal polynomials also have all real roots. We discuss this in more detail in section 4. }

\subsection*{Branch conventions}
Throughout, $\log z$ denotes the principal branch of the logarithm, with
$\arg(z)\in(-\pi,\pi]$. The square root is defined by
\[
\sqrt{z}=\exp\!\left(\frac{1}{2}\log z\right),
\]
and therefore uses the same branch cut $(-\infty,0]$. The dilogarithm
$\operatorname{Li}_2(z)$ is defined by analytic continuation of its power
series on $|z|<1$ with branch cut $[1,\infty)$. All identities in this
paper are understood with respect to these conventions. Expressions such as $\log(e^{i\theta})$ are always
interpreted using the principal determination, so that
$\log(e^{i\theta})=i\theta$ for $\theta\in(-\pi,\pi]$.

\subsection{Comparison with existing literature}

\noindent If for \eqref{rogers} we let $x,y$ be functions of a single variable, $u$, we can iterate \eqref{rogers} to create functional identities  such that :
\begin{equation}
\begin{aligned}
\sum_{j=1}^N A_jL(f_j(u)) \in \pi^2 \mathbb{Q} 
\end{aligned}
\end{equation}
 
\noindent The usual method  of producing functional identities is through iteration of \eqref{rogers}.  By expressing $x,y$ as functions of $u$, we can enumerate arrays, which we define as $a_n$, of a 5-term relation, which for suitable choices of $x,y$, multiple arrays $(a_1,a_2...a_n)$ can under favorable conditions collapse into a ladder or other relation. However, proving and finding new ladder identities and functional equations is difficult, as iterating \eqref{rogers} produces many unique elements $f_j(u)$, with no clear way to group them into anything fruitful, or which substitutions or ordering of iterations to use. 
\\
\\
For example, Kirillov \cite{Kirillov1995} and Lewin \cite{Lewin1991text} independently derived:

\begin{equation}\label{eq:ladders}
\begin{aligned}
L\!\left(-z^7 \frac{1-z}{1+z}\right)
&= 2L\!\left(z^2(1-z)\right)
 + L\!\left(\frac{-z^3}{1-z^2}\right)
 + 2L\!\left(\frac{z^3}{1+z}\right) \\
&\quad + L\!\left(-z(1-z^2)\right)
 + \frac{7}{4}L(z^4)
 - \frac{9}{4}L(z^2)
 + \frac{1}{2}L\!\left(z\frac{1-z}{1+z}\right)
 - \frac{1}{2}L\!\left(-z\frac{1+z}{1-z}\right)
\end{aligned}
\end{equation}

\noindent

\noindent Kirillov used \eqref{eq:ladders}, along with other relations, to prove the quadratic-base ladders of Coxeter and Browkin in
$\mathbb{Q}(\sqrt{5})$ and $\mathbb{Q}(\sqrt{13})$, respectively. Lewin used \eqref{eq:ladders} to find a ladder relation with the base equation $1-w=w^5$, and dejectedly concluded it ``otherwise it leads to nothing new." Indeed, there does not seem to be any symmetry or other pattern to \eqref{eq:ladders}, limiting their utility. 
\\
\\
Moreover, numerous papers connect special values of ${}_3F_2$ and ${}_4F_3$
hypergeometric series with dilogarithmic evaluations, producing rapidly convergent
representations of polylogarithmic constants rather than bona fide functional identities. For example, Z.-W.\ Sun and Y. Zhou \cite{SunZhoupreprint},  (in 2024, at around the same time of the original draft of this paper) similarly derived relations between ${}_4F_3$ and the dilogarithm in fields $\mathbb{Q}(u,\sqrt{4-3u^2})$ $\mathbb{Q}(u,\sqrt{u(4-3u)})$. Other papers find integrals that evaluate to various dilogarithm relations. For example \cite{BatirDasireddy}:

\begin{align}
\int_{0}^{1}\frac{\log\big(u^2+q\big)}{1+u}{\rm d}u=\log 2\log(1+q)-\arctan^2\left(\frac{1}{\sqrt{q}}\right)+\frac{1}{2}\operatorname{Li}_2\left(\frac{1}{q+1}\right).
\end{align}

\noindent But the idea of creating a functional dilogarithm identity by relating a ``dilogarithm-producing integral" to a so-called ``partner integral" appears to be new. So this paper represents  to the best of our knowledge a new way of creating functional higher logarithmic   identities.

\subsection{Procedure}

Our framework is as follows:
\\
\\
1. We show that the ratio of two cubic- and sextic-degree integrals, denoted by $\frac{w_1}{w_2}$ equals a rational number, by converting  said integrals into equivalent ${}_{4}F_{3}$-series. because the series differ by a rational constant, by being shifted, we can equate the integrals.
\\
\\
2. By factoring and evaluating said integrals, and then combining elements by exploiting certain symmetry relations and \eqref{duplicationformula}--\eqref{wellknown4term}, it's possible to derive 3- and 6-term functional identities of the form $\mathbb{Q}(u,\sqrt{f(u)})$,  . 
\\
\\
3. We exploit the quadratic field structure of the factorization to construct new ladders and prove existing ones, in addition to finding other relations.

\section{Proofs of 3- and 6-term functional identities }\label{sectionsextic}

\subsection{Proof of Theorem 1}

Define 
\begin{align}
w_1 &= \int_0^1 \frac{\arctan\left(h(1-y)\sqrt{y}\right)}{2hy}\,dy 
= \int_{0}^{1} \frac{(3x^2-1)\ln x}{1 + h^2 x^2(1-x^2)^2}\,dx 
, \label{eq:w1} \\[10pt]
w_2 &= \int_0^1 \frac{\arctan\left(h(1-y)\sqrt{y}\right)}{h(1-y)}\,dy 
= \int_{0}^{1} \frac{(1-3x^2)\ln(1-x^2)}{1 + h^2 x^2(1-x^2)^2}\,dx 
 \label{eq:w2}
\end{align}
\\
\\
And define $s_1$ \ddag: 
\begin{equation}\label{displays1}
 s_1 = \frac{2}{3}\sum_{n=0}^{\infty} 
 h^{2n} \left( -\frac{4}{27} \right)^{n} 
 \left[ \begin{matrix} \frac{1}{2}, \frac{1}{2}, 1 
 \vspace{1mm} \\ 
 \frac{5}{6}, \frac{7}{6}, \frac{3}{2} \end{matrix} \right]_{n}. 
\end{equation}
\\
 The key insight is to show that $w_1=w_2$. Dilogarithm functional identities follow via integration. We prove this equivalency through a ``shifting'' applied related ${}_{4}F_{3}$-series $s_1$, by showing how $s_1=w_1=w_2$.

\begin{proof}
 
Consider the generalized integral: 
\begin{equation}\label{integrandforsum1}
 \int_{0}^{h} \frac{(1 - y)^c y^d}{1 + g^2 (1 - y)^a y^b} \, dg = (1 - y)^{c - \frac{a}{2}} y^{d - \frac{b}{2}} 
 \operatorname{arctan}\left( h (1 - y)^{\frac{a}{2}} y^{\frac{b}{2}} \right) 
 \end{equation}
 for $\Re\left( \frac{ \left( 1 - y \right)^{-a/2} y^{-b/2} }{h} \right) \neq 0$. We refer to such integrals as \emph{sextic integrals}
 \\
 \\
For \eqref{eq:w1} and \eqref{eq:w2}, for appropriate $c,d,a,b$ of \eqref{integrandforsum1}, we have the power series expansions, respectively: 
   \begin{equation}\label{addedexplanation2}
  \int_{0}^{1}  \int_{0}^{h} \frac{1-y}{2 h \sqrt{y} \left(1 + g^2 (1-y)^2 y \right)} dg\ dy
 = \frac{1}{2h }  \int_{0}^{1}  \int_{0}^{h}  \sum_{n=0}^{\infty} (-1)^{n} g^{2n} y^{n - \frac{1}{2}} (1-y)^{2n+1}  dg\ dy 
\end{equation}

 \begin{equation}\label{addedexplanation2b}
  \int_{0}^{1}  \int_{0}^{h} \frac{\sqrt{y}}{2 h  \left(1 + g^2 (1-y)^2 y \right)} dg\ dy
 = \frac{1}{2h }  \int_{0}^{1}  \int_{0}^{h}  \sum_{n=0}^{\infty} (-1)^{n} g^{2n} y^{n+1/2 } (1-y)^{2n}  dg\ dy 
\end{equation}

\noindent For \eqref{addedexplanation2}, the LHS inner integral evaluates as such:

\begin{equation}\label{addedexplanation1}
 \int_{0}^{h} 
 \frac{1-y}{2 h \sqrt{y} \left(1 + g^2 (1-y)^2 y \right)} \, dg =\frac{ \operatorname{arctan}\left( h (1 - y) \sqrt{y} \right) }{2 h y} 
\end{equation}

\noindent Perform term-wise integration on the inner integral of the RHS. Then use the beta integral \dag \ to convert the outer integral of the RHS into series:

\begin{equation}
\int_{0}^{1}\int_{0}^{h} \frac{1-y}{2h\sqrt{y}\left(1+g^2(1-y)^2 y\right)}\,dg\,dy 
= \frac{1}{2}\sum_{n=0}^{\infty} \frac{(-1)^{n} h^{2n}}{2n+1}
\cdot\frac{\Gamma\!\left(n+\frac{1}{2}\right)\,\Gamma(2n+2)}{\Gamma\!\left(3n+\frac{5}{2}\right)},
\qquad |h| < \frac{3\sqrt{3}}{2}.
\end{equation}

\noindent Simplification yields $s_1$. Repeating this procedure for \eqref{addedexplanation2b} shows how the hypergeoemtric series are equal.  
 \\
 \\
 Define $h = \frac{i}{u(1-u^2)}$, the denominator of \eqref{eq:w1} and  \eqref{eq:w2} factors as:  
\begin{equation}\label{productforpartial}
 u^2(1-u^2)^2-x^2(1-x^2)^2=(u^2+ux+x^2-1)(u^2-ux+x^2-1)(u^2-x^2) 
\end{equation}

\noindent Setting $w_1=w_2$ and using partial fraction decomposition on \eqref{eq:w1} and  \eqref{eq:w2} and evaluating the integrals, a long  calculation and using \eqref{wellknown4term} and elementary transformations to combine terms and exploit symmetries (see appendix), proves \Cref{firstformatladder}.   
\end{proof}

\noindent \dag More information about the ``beta technique" can be found in \cite{Batir2005} and \cite{Hakimoglupreprint} .

\noindent \ddag For the  ${}_{4}F_{3}$-series, we make use of the notational shorthand: 
\begin{equation*}
 \left[ \begin{matrix} \alpha, \beta, \ldots, \gamma \vspace{1mm} \\ 
 A, B, \ldots, C \end{matrix} \right]_{n} = \frac{ (\alpha)_{n} (\beta)_{n} 
 \cdots (\gamma)_{n} }{ (A)_{n} (B)_{n} \cdots (C)_{n}}. 
\end{equation*}

\subsection{Proof of Theorem 2.}

\begin{proof}
For \eqref{rogers}, let $$ w = \frac{2u^2 - 3u + \sqrt{(4 - 3u)u}}{2(u - 1)^2} \ \ \ \text{and} \ \ \ 
 z = \frac{2u^2 - 3u - \sqrt{(4 - 3u)u}}{2(u - 1)^2}. $$ By then applying the Landen identity, we obtain that 
 $$ \operatorname{Li}_2 \left(\frac{u + \sqrt{(4 - 3u) u}}{2 (u - 1)}\right) + \operatorname{Li}_2\left(\frac{u - 
 \sqrt{(4 - 3u) u}}{2 (u - 1)}\right) = -\frac{1}{2} \ln^2\left(1 - \frac{u + \sqrt{(4 - 3u) u}}{2 (u - 1)}\right). $$ 
 Similarly, for $ |u|<1$, we obtain the vanishing of 
\begin{multline*}
 \ln\left(\frac{-u + \sqrt{u (4-3u)} + 2}{2 - 2 u}\right) \ln\left(\frac{u + \sqrt{u (4-3u)} - 2}{2 (u - 1)}\right) + \\ 
 \ln^2\left(1 - \frac{u + \sqrt{(4 - 3 u) u}}{2 (u - 1)}\right) = 0. 
\end{multline*}
 This may be obtained by setting $$ a = \frac{2 - u + \sqrt{u(4 - 3u)}}{2 - 2u} \ \ \ \text{and} \ \ \ b = \frac{-2 + u + \sqrt{u(4 - 
 3u)}}{-2 + 2u}, $$ since the above equalities yield $(\ln(a)+\ln(b))^2-\ln^2(a)-\ln(a)\ln(b)=\ln^2(b) + 
 \ln(a)\ln(b)$. Because $ab=1$, we have that $(\ln(a)+\ln(b))^2=0$, so we have that $0= \ln^2(b)+\ln(a)\ln(b) + 
 \ln^2(a)+\ln(a)\ln(b)=(\ln(a)+\ln(b))^2=0$. 
\end{proof}

\noindent Our justification for the specific $w,z$ chosen above follows from the cubic integrals: 
\begin{equation}\label{w1w2forCampbell}
w_1 = \int_{0}^{1} \frac{\ln\left(1 - \frac{1}{u(1-u)^2} x(1-x)^2\right)}{1-x} \, dx \text{ and } w_2 = \int_{0}^{1} \frac{\ln\left(1 - \frac{1}{u(1-u)^2} x(1-x)^2\right)}{x} \, dx 
\end{equation}
 along with 
\[ -\frac{z}{6} \sum_{n=0}^{\infty} \left( \frac{4z}{27} \right)^n 
 \left[ \begin{matrix} 
 1, 1, \frac{3}{2} 
 \vspace{1mm} \\ 
 \frac{4}{3}, \frac{5}{3}, 2 \end{matrix} \right]_{n} = 
 \int_0^1 \frac{\ln(1 - z x (1 - x)^2)}{x} \, dx. 
\]
  Using the shifting and factoring method discussed earlier, we have $\frac{w_2}{w_1}=2$, and obtain:
  
  $$ \operatorname{L}\left( \frac{ 2-\sqrt{4u - 
 3 u^2} -u }{2} \right) + \operatorname{L}\left( \frac{2+\sqrt{4u-3 u^2} -u}{2} \right) + 
 \operatorname{L}(u) \in \pi^2 \mathbb{Q}$$

\noindent Applying the Landen transformation to $ \frac{ 2-\sqrt{4u - 
 3 u^2} -u }{2} $, we obtain  $\frac{2u^2 - 3u - \sqrt{4u - 3u^2}}{2(u - 1)^2}$.

\remark Evaluating \eqref{integrandforsum1} and \eqref{w1w2forCampbell} at other $a,b$ or exponents either produces 
polynomials that do not factor over quadratic fields, or are trivial. Consequently, they lack the
symmetry and reflection properties needed to generate dilogarithm ladders or to compress
into concise functional identities. Although identities may still exist in such cases,
they often involve many terms rather than the six-term relations of \Cref{firstformatladder} or three-term relations of \Cref{previouslymaster}, with no
apparent mechanism for further reduction. In other instances, the resulting identities
are degenerate and fail to yield genuine dilogarithmic relations. Or reduce to
Theorems~1 and~2 through elementary transformations that produce no new
results.
\section{Applications} 

\subsection{A full solution to a problem due to Sun}
 As it turns out, the $u = \frac{1}{\sqrt{3}}$ case Theorem \ref{maintheorem} can be used to formulate a full solution to a problem due 
 to Sun \cite{Sun2021,Sun2015} that was partially solved by Campbell \cite{Campbell2023}, as later explained. 

\begin{theorem}\label{theoremSunCampbell}
 The closed-form evaluation
 $$ 8 \operatorname{Li}_2\left(2-\sqrt{3}\right)-\operatorname{Li}_2\left(4 \sqrt{3} - 7 \right) 
 = \frac{5 \pi ^2}{12}+\ln \left(2-\sqrt{3}\right) \ln \left(2+\sqrt{3}\right) $$
 holds true.
\end{theorem}

\begin{proof}
 For the $u = \frac{1}{\sqrt{3}}$ case Theorem \ref{maintheorem}, 
 the left-hand side of \eqref{mainABCD} 
 may be expressed in terms of 
 $ -2 \operatorname{Li}_2 \left( 7 - 4 \sqrt{3} \right) + 
 4 \operatorname{Li}_2 \left( \sqrt{3} - 2 \right) $ and a closed-form expression, 
 and this follows from the known relation 
\begin{equation*}
-2 \operatorname{Li}_2 \left( \frac{3 - \sqrt{3}}{6} \right) = 2 \operatorname{Li}_2 \left( \sqrt{3} - 2 \right) + 
 \ln^{2} \left( \frac{3 + \sqrt{3}}{6} \right). 
\end{equation*}
 The right-hand side of \eqref{mainABCD}, for the same input value 
 $u = \frac{1}{\sqrt{3}}$ may be rewritten as 
 $$ K = \operatorname{Li}_2\left(-\sqrt{3}\right)-\frac{1}{2}
 \operatorname{Li}_2\left(-\frac{\sqrt{3}}{2}\right) + 
 \frac{\operatorname{Li}_2\left(\frac{\sqrt{3}}{2}\right)}{2}-\operatorname{Li}_2\left(\sqrt{3}\right). $$
 Using the well known $4$- term functional equation 
\begin{equation}\label{wellknown4termB}
 \operatorname{Li}_2(z) - \operatorname{Li}_2(-z) + 
 \operatorname{Li}_2\left(\frac{1 - z}{1 + z}\right) - 
 \operatorname{Li}_2\left(\frac{-(1 - z)}{1 + z}\right) = \frac{\pi^2}{4} + \ln(z) \ln\left(\frac{1 + z}{1 - z}\right), 
\end{equation}
 we obtain, by setting $z = \frac{\sqrt{3}}{2}$ and $z=\sqrt{3}$, that the right-hand side of \eqref{mainABCD} may be rewritten as 
\begin{equation*}
 \operatorname{Li}_2\left(7 - 4 \sqrt{3}\right) - \operatorname{Li}_2\left(4 \sqrt{3} - 7 \right) - 
 2 \, \operatorname{Li}_2\left(\sqrt{3} - 2\right) + 2 \, \operatorname{Li}_2\left(2 - \sqrt{3}\right). 
\end{equation*}
 Equating our equivalent expressions for the right-hand and left-hand sides of \eqref{mainABCD}, we obtain 
\begin{equation*}
 3 \, \operatorname{Li}_2\left(7 - 4 \sqrt{3}\right) - 
 \operatorname{Li}_2\left(4 \sqrt{3} - 7\right) - 
 6 \, \operatorname{Li}_2\left(\sqrt{3} - 2\right) + 2 \, \operatorname{Li}_2\left(2 - \sqrt{3}\right)
\end{equation*}
 and 
\begin{equation*}
 3 \, \operatorname{Li}_2\left(7 - 4 \sqrt{3}\right) = 6 \, 
 \operatorname{Li}_2\left(2 - \sqrt{3}\right) + 6 \, \operatorname{Li}_2\left(\sqrt{3} - 2\right), 
\end{equation*}
 and this gives us an equivalent version of the desired result. 
\end{proof}

\noindent  It was shown by Campbell \cite{Campbell2023} that the open problem due to Sun \cite{Sun2021,Sun2015} given by proving the 
 conjectured formula 
\begin{equation}\label{Sunconjectureharmonic}
 \sum_{n=0}^{\infty} \frac{ \binom{2n}{n} }{ (2n+1) 8^{n} } \left( \sum_{0 \leq k < n} 
 \frac{(-1)^k}{2k+1} - \frac{(-1)^n}{2n+1} \right) = -\frac{\sqrt{2}}{16} \pi^2 
\end{equation}
 is equivalent to the problem of proving the purported closed-form evaluation 
\begin{equation}\label{discoveredwithAu}
 \operatorname{Li}_{2}\left( \frac{\sqrt{3} + 2}{4} \right) - 8 \operatorname{Li}_{2}(2 - \sqrt{3}) 
 = -\frac{\pi^2}{4} - 2 \ln^{2}(2) + \frac{5}{2} \ln^2(\sqrt{3} + 2) - 2 \ln(\sqrt{3} + 2) \ln 2 
\end{equation}
 discovered experimentally using a Mathematica package given by Au, and referring to Campbell's work \cite{Campbell2023} for details. 
 It can be shown that Theorem \ref{theoremSunCampbell} is equivalent to \eqref{discoveredwithAu}, thus providing a full proof of Sun's 
 conjectured formula in \eqref{Sunconjectureharmonic}. 

\subsection{A relation between $\mathbb{Q}[\sqrt{13}]$,  $\mathbb{Q}[\sqrt{3}]$ , and $\mathrm{Cl}_2\!\left(\tfrac{\pi}{3}\right)$}\label{subsec:sqrt13cl2}

We find  a relation between ${Q}[\sqrt{13}]$ and ${Q}[\sqrt{3}]$ that gives a fast series for $\rm{Cl}_2\big(\tfrac{\pi}3\big)$ that bypasses the summation of the multiple angle formula.  
\\
\\
First, by setting $u=  i\sqrt{3}$ in \eqref{eq:w1}, we have:
\begin{equation}\label{eq:u-isqrt3}
\frac{1}{6\sqrt{3}}\sum_{n=0}^{\infty}  \left(  \frac{-1}{324}  \right)^{n} \frac{ (1/2)_n(1/2)_n(1)_n}{(3/2)_n(7/6)_n(5/6)_n}   =4 \sqrt{3} \int_{0}^{1} \frac{(3x^2 - 1)\ln(x)}{48 + x^2(1 - x^2)^2} \, dx
\end{equation}

\noindent Using Theorem \ref{maintheorem} , for the RHS of  (where $\Im$ is the imaginary part), we have:
\begin{equation}
\\\Im \left(\text{Li}_2\left(\frac{i}{\sqrt{3}}\right) - \text{Li}_2\left(\frac{-i}{\sqrt{3}}\right)\right)=   \frac{5}{3} \, \Im\text{Li}_2\left(\frac{1 + i\sqrt{3}}{2}\right) - \frac{\pi \ln(3)}{6}  
\end{equation}
\\
On the RHS, using \eqref{wellknown4term} setting, $ x=2/(i \sqrt{3}-\sqrt{13})$, we obtain:
\begin{equation}
\begin{aligned}
\text{Li}_2\left(-\frac{1}{12} \left(3 + i\sqrt{3}\right)\left(5 + \sqrt{13}\right)\right) - \text{Li}_2\left(\frac{1}{12}\left(3 + i\sqrt{3}\right)\left(5 + \sqrt{13}\right)\right)
\ \\
= \text{Li}_2\left(-\frac{2i + \sqrt{3} + i\sqrt{13}}{-2i + \sqrt{3} + i\sqrt{13}}\right) - \text{Li}_2\left(\frac{2i + \sqrt{3} + i\sqrt{13}}{-2i + \sqrt{3} + i\sqrt{13}}\right)
\end{aligned}
\end{equation}
\\
\\
Putting the LHS and RHS together. Starting from the identity
\begin{align*}
&\frac{\pi}{12}\ln\!\left(\frac{4}{3}\right)
+\cosh^{-1}\!\left(\frac{19}{6}\right)\tan^{-1}\!\left(\frac{\sqrt{13}-2}{\sqrt{27}}\right)
+2\,\Im\,\Li_2\!\left(\frac{5-\sqrt{13}}{\sqrt{12}}\,e^{i\pi/6}\right)
-\Cl_2\!\left(\frac{\pi}{3}\right) \\[4pt]
&\quad=\Im\,\Li_2\!\left(\frac{\sqrt{13}-5}{\sqrt{12}}\,e^{i\pi/6}\right)
-\Im\,\Li_2\!\left(\frac{5-\sqrt{13}}{\sqrt{12}}\,e^{i\pi/6}\right)
+\frac{\pi}{6}\ln 2 \\[4pt]
&\qquad-\frac12\tan^{-1}\!\left(\sqrt{\frac{3}{13}}\right)
\ln\!\left(\frac{19+5\sqrt{13}}{6}\right)
+\frac56\,\Cl_2\!\left(\frac{\pi}{3}\right)
-\frac{\pi}{12}\ln 3,
\end{align*}
one uses $\frac{\sqrt{13}-5}{\sqrt{12}}e^{i\pi/6}=-z$ and, writing $A=\Im\,\Li_2(z)$, $B=\Im\,\Li_2(-z)$, collects all dilogarithm terms on one side.  Moreover, using $\cosh^{-1}x=\ln\!\left(x+\sqrt{x^2-1}\right)$ one finds
\[
\cosh^{-1}\!\frac{19}{6}=\ln\!\left(\frac{19+5\sqrt{13}}{6}\right),
\]
so both arctangent terms multiply the same logarithm, and they combine via the elementary sub-identity
\[
\tan^{-1}\!\left(\frac{\sqrt{13}-2}{\sqrt{27}}\right)+\frac12\tan^{-1}\!\sqrt{\frac{3}{13}}=\frac{\pi}{6}.
\]
All logarithms cancel exactly, leaving the closed form
\begin{equation}
\label{eq:dilog-duplication}
3\,\Im\,\Li_2(z)-\Im\,\Li_2(-z)=\frac{11}{6}\,\Cl_2\!\left(\frac{\pi}{3}\right)-\frac{\pi}{6}\,\cosh^{-1}\!\left(\frac{19}{6}\right),
\qquad z=\frac{5-\sqrt{13}}{\sqrt{12}}\,e^{i\pi/6}.
\end{equation}

\noindent To obtain an infinite series without any dilogarithms, we can use $\text{Li}_2(z)+\text{Li}_2(-z)=1/2\text{Li}_2(z^{2})$ to compress the two dilogarithms into a single one. This speeds up the convergence by squaring, because by taking the imaginary part of  $e^{\frac{n\pi i}{3}}$ for positive integers $n$, every third term is ignored. Evaluating \eqref{eq:u-isqrt3} and combining \eqref{eq:dilog-duplication} we have:
\begin{equation}
\begin{aligned}
 \pi \ln\left(\frac{4}{3}\right) + 3\ln\left(\frac{1}{6}(19 + 5\sqrt{13})\right)\left( \tan^{-1}\left(\frac{3\sqrt{3}}{16 + 5\sqrt{13}}\right)\right)
+ \\
3\cdot\Im\left[\text{Li}_2\left(\frac{1}{12}(5 - \sqrt{13})^2 e^{\frac{\pi i}{3}}\right)\right] - \rm{Cl}_2\big(\frac{\pi}3\big)
=\frac{2}{\sqrt{3}}\sum_{n=0}^{\infty}  \left(  \frac{-1}{324}  \right)^{n} \frac{ (1/2)_n(1/2)_n(1)_n}{(3/2)_n(7/6)_n(5/6)_n}
\end{aligned}
\end{equation}
And:
\begin{equation}
\begin{aligned}
\Im\left[\text{Li}_2\left(\frac{1}{12}(5 - \sqrt{13})^2 e^{\frac{\pi i}{3}}\right)\right]=\sqrt{3}\sum_{n=1}^{\infty} \frac{(19 - 5 \sqrt{13})^{3n - 2} }{(-216)^n} \left( \frac{5 \sqrt{13} - 19}{(1 - 3n)^2} - \frac{6}{(2 - 3n)^2} \right)
\end{aligned}
\end{equation}

\subsubsection{Infinite series for  $\pi^2$ }

\noindent As an addendum, for example, the connection between $\sqrt{3}$ and $\sqrt{13}$ arises again:
\begin{equation}
\ln^2\left(\frac{5 + \sqrt{13}}{2\sqrt{3}}\right) - \frac{\pi^2}{12}=\frac{1}{180}\, _4F_3\left({1,1,1,\frac{3}{2}\atop 2,\frac{11}{6},\frac{7}{6}};\frac{-1}{324}\right)
\end{equation}

\noindent Which follows from derived  with a modified version of the integral procedure in section 2. (For $h \geq \alpha >0 $ where $\frac{4}{27} \cdot \frac{1}{\alpha^2 (1 + \alpha^2)} = 1$):
\begin{multline}
\frac{4}{15h^2(1+h^2)^2}\, _4F_3\left({1,1,1,\frac{3}{2}\atop 2,\frac{11}{6},\frac{7}{6}};\frac{-4}{27h^2(1+h^2)^2}\right)=
\\
\frac{-3}{4} \left(\pi - 2\arctan(h)\right)^2 
+ \ln^2\left(\frac{h^2 + 2 + \sqrt{3h^2 + 4}}{h\sqrt{1 + h^2}}\right)
\end{multline}

\subsection{Proving Bytsko's 2-term identities}

In the context of the Nahm conjecture, Zagier \cite{Zagier2007} numerically showed $\xi_Q = [\alpha] + [\alpha\beta^3] $ is a torsion element, where $\alpha, \beta \in (0,1)$ are solutions of the system:
\begin{align*}
1 - \alpha &= \alpha^3 \beta, \\
1 - \alpha\beta^3 &= \alpha \beta^2.
\end{align*}  This is equivalent to solving $x^6 + x^5 - 2 x^3 + 2 x - 1 =0$. Here, $y=1/(1-x_1)$, where $x_1$ is the negative real root and $x=x_0$ is the positive real root. Bytsko 
 \cite{Bytsko1999} conjectured the related 
 associated two-term dilogarithm identity $\operatorname {L}(x)+\operatorname{L}(y)=\frac{4\pi^2}{21}$ and asked if it's related to the trio of Watson $\pi/7$ identities. We answer in the affirmative.

\begin{proof} Applying \eqref{rogers} multiple times to \Cref{firstformatladder}:

\begin{lemma}

\begin{multline*}
 -3\operatorname{L}\left( \frac{-u - \sqrt{4 - 3u^2} }{2} \right) 
 - 3 \operatorname{L}\left(\frac{-u + \sqrt{4 - 3u^2} }{2}  \right) 
 -  \operatorname{L}\left( \frac{1 - u}{u + 1} \right)  +  \operatorname{L}\left( \frac{u-1}{u + 1} \right) \\ 
 + \operatorname{L}\left( \frac{u^2 +u- \sqrt{4 - 3u^2} + 2}{2u} \right) 
+ \operatorname{L}\left( \frac{u^2+u + \sqrt{4 - 3u^2} + 2}{2u} \right) 
  -3 \operatorname{L}(u)\in \pi^2 \mathbb{Q}  , 
\end{multline*}
\end{lemma}
\noindent Solving \begin{align*}
\frac{u + 2 - 2u^2 + \sqrt{4 - 3u^2}}{2u} &= y, \\
\frac{u + 2 - 2u^2 - \sqrt{4 - 3u^2}}{2u} &= \frac{1}{y}.
\end{align*} 

\noindent We have $u=2\cos(3\pi/7)$. Using the Watson identity $2[u]+[u^2]=0$  , and applying \eqref{wellknown4term} and \eqref{Landenformula} , we have  $ -\left[ \frac{1 - u}{u + 1} \right]  +   \left[ \frac{u-1}{u + 1} \right] -3 [u]=0$, and $[y]+[1/y]=0$ 
\\
\\
\noindent Hence, we obtain  for $x_1,x_0$:
\begin{equation}
\pm\sqrt{1 - 3\sin^2\left(\frac{\pi}{14}\right)} - \sin\left(\frac{\pi}{14}\right)
\end{equation}
\end{proof}
\subsection{A two-term identity over $\mathbb{Q}(\sqrt{2})$}
 Bytsko \cite{Bytsko1999} also conjectured a two-term identity over $\mathbb{Q}(\sqrt{2})$ through asymptotic analysis. As 
 above, Bytsko noted that a proof is lacking. 

 \[
L\!\left(\frac{3}{2} - \frac{1}{2}\sqrt{2} - \frac{1}{2}\sqrt{2\sqrt{2}-1}\right)
+
L\!\left(\left(\frac{3}{2} + \sqrt{2}\right)\sqrt{2\sqrt{2}-1} - \frac{3}{2} - \frac{3}{2}\sqrt{2}\right)
=
\frac{\pi^2}{12}.
\]

\begin{proof}
    
 Consider the identity:
\begin{lemma}\label{previously10p2}
 \begin{multline}
 \operatorname{L}(k) - \operatorname{L}(k^2) 
 - \operatorname{L}\left(\frac{-2k^2 + k - \sqrt{-3k^2 + 2k + 1} + 1}{2 - 2k^2}\right) - \\ 
 \operatorname{L}\left(\frac{-2k^2 + k + \sqrt{-3k^2 + 2k + 1} + 1}{2 - 2k^2}\right)  \in \pi^2 \mathbb{Q}
\end{multline}
\end{lemma}

\noindent This follows from \Cref{previouslymaster} by applying elementary  transformations found in \Cref{sec:prelim}, and setting $u = 1 - k$. Iterating with \eqref{rogers}, we obtain a 5-term relation, composed of: $$ z(k) = \frac{k^2 + (k + 1) \sqrt{-3k^2 + 2k + 1} - 
 2k - 1}{2(k^2 - k - 1)} $$ and $$ w(k) = \frac{k^2 - (k + 1) \sqrt{-3k^2 + 2k + 1} - 2k - 1}{2(k^2 - k - 1)}. $$ Now, we have to 
 show that for $\operatorname{Li}_2(k) - \operatorname{Li}_2(k^2) - \operatorname{Li}_2\left( \frac{k^3}{(k - 1)(k + 1)^2} \right)$ the 
 dilogarithm parts cancel, in which we let $k=\sqrt{2}$. Substituting into $z(k)$, we have $\frac{1}{2} \left( 3 - \sqrt{2} - \sqrt{2\sqrt{2} - 
 1} \right) $, which is a root of $x^4 - 6 x^3 + 13 x^2 - 10 x + 1=0$. This agrees with Bytsko's result.

\end{proof}

\subsection{8-degree 2-term identity}

We extend Bytsko's result for complex values. Roger's L function has a complex counterpart, called the Bloch-Wigner function, defined as 

\begin{equation}
\operatorname{D}_2 (z) = \operatorname{Im} (\operatorname{Li}_2 (z) )+\arg(1-z)\log|z| , \ \text{if}  \ z \in \mathbb{C} \setminus \{0, 1\}.
\end{equation}

\noindent In special circumstances, there exist non-trivial 2-term closed-form relations. Known examples include:
$\operatorname{Li}_2\left(\frac{1}{2} - \frac{i}{2}\right) = -i G + \frac{5\pi^2}{96} - \frac{\ln^2 2}{8} + \frac{i \pi \ln 2}{8}$, where G is Catalan's constant. 
\\
\\
\noindent For  $y^8 - 3 y^7 + y^6 + 4 y^5 - 6 y^3 + 3 y + 1=0$, where approximately  $r_1=-0.77367 - 0.62681 i$ and $r_2=1.32654 + 0.31844 i$ (checked to 50 digits), the below identity holds:
\begin{equation}
\operatorname{D}(r_1)+\operatorname{D}(r_2)=0
\end{equation}

\noindent Similar to 3.6, the proof follows from Theorem  2, but by a different matching condition. 

\begin{lemma}\label{previously10p2b}
 The equality of 
\begin{multline*}
 \operatorname{Li}_2 \left( \frac{2 u^2 - 3 u - \sqrt{u(4-3u)}}{2 (u-1)^2} \right) 
 + \operatorname{Li}_2 \left( \frac{(u - 1) \left(2 u^3 - 4 u^2 + 3 u - \sqrt{u(4 - 
 3 u)} \right)}{2 \left(u^4 - 3 u^3 + 4 u^2 - 2 u + 1\right)} \right) \\ 
 + \operatorname{Li}_2 \left( \frac{u \left( u^2 - u + (u - 1) \sqrt{u(4-3u)} + 
 2 \right)}{2 \left(u^4 - 3 u^3 + 4 u^2 - 2 u + 1\right)} \right) 
 + \operatorname{Li}_2 \left( \frac{u \left(2 u^2 - 3 u + \sqrt{u(4-3u)}\right)}{2 (u-1)^2} \right) 
\end{multline*}
 and the closed form 
\begin{multline*}
 \text{ { \footnotesize $ - \ln \left( \frac{u^2 - u + (u - 
 1) \sqrt{u(4 - 3 u)} + 2}{2 \left(u^4 - 3 u^3 + 4 u^2 - 2 u + 1\right)} \right) 
 \ln \left( \frac{(u - 1) \left(2 u^3 - 5 u^2 - u \sqrt{(4 - 3 u) u} + 4 u - 2\right)}{2 \left(u^4 - 3 u^3 + 
 4 u^2 - 2 u + 1\right)} \right) - $ } } \\ 
 \text{ { \footnotesize $ \frac{3}{2} \ln^2 \left( 1 - \frac{u - \sqrt{u(4-3u)}}{2(u-1)} \right) - 
 \frac{1}{2} \ln^2 (1 - u)$. } }
\end{multline*}
 holds for complex $u$. 
\end{lemma}

And setting:

\begin{equation*}
\left\{
\begin{aligned}
\frac{(u-1) \left(2 u^3 - 4 u^2 + 3 u - \sqrt{u(4 - 3 u)}\right)}{2 \left(u^4 - 3 u^3 + 4 u^2 - 2 u + 1\right)} = z\\
\frac{u \left(u^2 - u + (u-1) \sqrt{u(4-3u)} + 2\right)}{2 \left(u^4 - 3 u^3 + 4 u^2 - 2 u + 1\right)} = \frac{1}{z}
\end{aligned}
\right.
\end{equation*}

\noindent A resultant calculation yields the sought 8-degree polynomial. 

\section{Dilogarithm Ladders}
As in this work  of \cite{Lewin1958,AbouzahraLewinXiao1987,Lewin1986,Lewin1981,Lewin1991text,Lewin1993,Lewin1991collection,CohenLewinZagier1992,Gangl2013}, we consider so-called \emph{ladders}, defined as :

\begin{equation}
\begin{aligned}
\sum_{j=1}^N A_jL(x^j) \in \pi^2 \mathbb{Q} 
\end{aligned}
\end{equation}
\noindent Where $A_j \in\mathbb{Q}$ and for algebraic $x \in (0,1)$. The minimal polynomial for $x$ is sometimes denoted as the \textit{base equation}.  
\\
\\
\noindent To present an example, for quadratic fields, Browkin \cite{Brow} proposed the following ladders:

\begin{align*}
L(x^{6})-6L(x^{3})+L(x^{2})+18L(x)
    &= 8L(1),  \\
L(z^{6})-3L(z^{3})-6L(z^{2})+9L(z)
    &= 2L(1). 
\end{align*}

\noindent where:
\[
x=\frac{\sqrt{13}-1}{6},
\qquad
z=\frac{\sqrt{13}+1}{6},
\]

\noindent The base equations are thus $3 x^2 + x - 1=0, 3 z^2 - z - 1=0$.
\\
\\
\noindent In subsequent sections, we classify  known non-trivial ladder base families, derive an analytic proof of a  ladder relation by Loxton-Lewin using \Cref{previouslymaster}, and derive a pair quartic-base  ladders, which represents a new ladder family using \Cref{firstformatladder}. 

\remark Higher order ladders also exist (e.g. trilogarithm ladders) but this  work focuses on the dilogarithm.

\subsection{Radius lemma}

In the context of ladders, we also introduce what we call the Radius Lemma, which we will use for later proofs. The radicals appearing in Theorems~1 and~2 take the form $a_1(u)+a_2(u)\sqrt{d(u)}$, with no evident way to convert them into the powers $r,\; r^2,\; r^3,\;\ldots$ that underlie classical polylogarithmic ladders. Since these radicals occur as conjugate pairs, we exploit this symmetry through what we call the
\emph{radius method}, which identifies exceptional values of \(u\) for which
the radical expressions compress to a single parameter \(r\) suitable for
ladder constructions.
\\
\\

\noindent Consider the well-known identity (for real $r)$:

\begin{equation*}\label{eq:cuberoot}
    -\operatorname{Li}_2(r) + \frac{1}{3}\operatorname{Li}_2\!\left(r^3\right)
    = \operatorname{Li}_2\!\left(r\,e^{2\pi i/3}\right) + \operatorname{Li}_2\!\left(r\,e^{4\pi i/3}\right).
\end{equation*}

\noindent This becomes more useful in the context of ladders once a conjugate pair of algebraic functions can be tuned to sit exactly at the cube roots of unity, as follows.

\begin{lemma}[Radius lemma]\label{lem:radius}
Let
\begin{equation*}
    f_j(u) = \frac{a(u) \pm b(u)\sqrt{k(u)}}{h(u)}
\end{equation*}
denote a conjugate pair of functions, with $a, b, h, k$ polynomials (or rational functions) in $u$. Impose the constraint
\begin{equation}\label{eq:constraint}
    \frac{b(u)\sqrt{k(u)}}{a(u)} = \pm i\sqrt{3}.
\end{equation}
Then:
\begin{enumerate}
    \item[(i)] Constraint \eqref{eq:constraint} forces an auxiliary function $c(u)$ to vanish identically:
    \begin{equation*}
        c(u) = 0;
    \end{equation*}
    \item[(ii)] under \eqref{eq:constraint}, the norm of the conjugate pair reduces to
    \begin{equation*}
        r \;=\; \frac{\sqrt{a(u)^2 - b(u)^2 k(u)}}{h(u)} \;=\; \frac{2a(u)}{h(u)}.
    \end{equation*}
\end{enumerate}
Eliminating $u$ between $c(u) = 0$ and $r = 2a(u)/h(u)$ via the resultant
\begin{equation*}
    p(r) = \operatorname{Res}_u\!\left(c(u),\; r\,h(u) - 2a(u)\right)
\end{equation*}
then yields a polynomial relation $p(r) = 0$ satisfied by the radius $r$.
\end{lemma}

\subsection{Unique ladder families}

 Gordon and McIntosh \cite{GordonMcIntosh1997} remarked that producing ladders that are not of a ``trivial" base ; e.g.  $u^a+u^b\pm u^c=1$,  $u^a+u^b=1$ is ``very difficult".  Don Zagier\cite{Zagier1991} also remarked that finding examples of ladders is ``very difficult and requires great ingenuity to produce examples.
 \\
 \\
 Some recent papers (e.g. Campbell \cite{Campbell2025}, Alha \cite{Alha2025}) purport new ladders, but these actually reduce to known relations. For example Campbell found a ladder with a base equation $x^{10}+x^9+x^8+x^7+x^6+x^5-x^4-x^3-x^2-x-1=0$, but this is the same as $x^{11}-2x^5+1=0$, which follows from setting $a,b=5, c=11$ above. The confusion arises because it's possible to express a ladder-generating base polynomial in may ways through different derivations while still getting the same solutions, and these may be interpreted to be unique ladders, when they are not. 
 \\
 \\
 A good summary of results up to 2022 is given by Campbell \cite{Campbell2025}. \noindent Below are the known table of unique ladder families (that are not of the form $u^a+u^b\pm u^c=1$,  $u^a+u^b=1$):

\begin{center}
\begin{tabular}{c l l l l}
\toprule
\# & Author(s) & Degree & Field/Angle & \# of real roots\\
\midrule
1 & Browkin, Lewin & 2 &  $\mathbb{Q}(\sqrt{13})$,  $\mathbb{Q}(\sqrt{15})$,  $\mathbb{Q}(\sqrt{21})$ &2 \\
2 & Loxton, Watson, etc. & 3 & $\pi/7$, $\pi/9$,$\pi/18$ & 3 \\
3 & Author & 4 & $\mathbb{Q}(\sqrt{33})$,  $\mathbb{Q}(\sqrt{5})$,  $\mathbb{Q}(\sqrt{10})$ &4\\
4 & Gordon--McIntosh & 4 & $\mathbb{Q}(\sqrt{5})$ &2\\
5 & Lewin, Bailey etc.  & misc. & Salem numbers & 2\\
6 &  Rogers and Kummer  & misc. & misc. & 2\\
7 &  Lewin  & misc. & misc. & 2\\
8 & Ramanujan, Bailey et al.  & 1 & 1/2,1/3 & 1\\
9 & Campbell  & misc. & misc. & 1\\
\bottomrule
\end{tabular}
\end{center}

\remark {\#5 are rich in ladder relations. A famous example is a seventeenth-order polylogarithm ladder given by David H. Bailey and David J. Broadhurst. }

\remark{\#3 The $\mathbb{Q}(\sqrt{33})$ case is derived analytically in section 4. The $\mathbb{Q}(\sqrt{5})$ and $\mathbb{Q}(\sqrt{10})$} cases were found numerically, in section 5.

\remark{\#6 is the 15-term ladder relation with a base equation that satisfies $F(u) = u^s + 1 - \sum_{n=1}^{5}\left(u^{q_n} + u^{s-q_n}\right) = 0$ where $s = \frac{1}{2}\sum_{n=1}^{5}q_n$}.

\remark{\#7 has the base equation given by  $u^p+u^n-u^q+u^{n+m}-u^{q+p+m}=1$ which is derived from a 9-term relation, Enforcing $q>n$ gives a real root in $(0,1)$. Upon abstracting the factor $(1-u)$, has 2 real roots.}

\remark {\#8 are linear ladder relations of bases 1/2 and 1/3.} 

\remark {Other quadratic field base equations exist, those being in fields $\mathbb{Q}(\sqrt{N})$, $N=2,5,3,6$, but they follow from one of the trivial relations or from \#6.} 

\remark {\#9, in private communication for a forthcoming paper, reproduced with his permission, John M. Campbell found the following base equation and ladder. Let $r \in (0,1)$ be the real root of
$x^9 + x^6 - x^4 + x^2 - 1 = 0$.

\subsection{Cubic ladders involving $\pi/7$, $\pi/9$, and $\pi/18$}  

For ladders in which \(x\) is a root of a minimal cubic polynomial, there are two classical families associated with the trigonometric values \(\pi/7\) and \(\pi/9\). Watson's ladder identities are as follows:
\[
L(\alpha)-L(\alpha^2)=\frac17L(1),
\]
\[
2L(\beta)+L(\beta^2)=\frac{10}{7}L(1),
\]
\[
2L(\gamma)+L(\gamma^2)=\frac{8}{7}L(1).
\]

\noindent Where  $\alpha$, $-\beta$, and $-1/\gamma$ are  the roots of
\[
x^3+2x^2-x-1=0,
\]
and 
\[
\alpha=\frac12\sec\frac{2\pi}{7},
\qquad
\beta=\frac12\sec\frac{\pi}{7},
\qquad
\gamma=2\cos\frac{3\pi}{7},
\]

\noindent These are strait-forward to prove by applying a few applications of \eqref{rogers} and  \eqref{duplicationformula}. A proof can be found in  Kirillov \cite{Kirillov1995} .

\subsubsection{Analytic proof of $\pi/9$ Loxton-Lewin ladder}. 

\noindent For ladders involving trigonometric values of $\pi/9$, proofs are more difficult. J. H. Loxton conjectured the below trio $\pi/9$ ladder identities in the late `70s or early `80s. He \cite{Loxton} proved the first two by applying asymptotic analysis to Rogers-Ramanujan type partition functions of Slater, but it failed for \eqref{eq:loxton-mu}, which resisted a proof for a while.

\begin{align}
\operatorname{Li}_2(\kappa)
+\operatorname{Li}_2(\kappa^2)
-\tfrac{1}{3}\operatorname{Li}_2(\kappa^3)
  &= \frac{7\pi^2}{54} - \log^2\kappa,
  &\quad \kappa
  &= \frac{1}{2}\sec\!\left(\frac{\pi}{9}\right), \label{eq:loxton-kappa}\\[6pt]
\operatorname{Li}_2(-\lambda)
+\operatorname{Li}_2(\lambda^2)
-\tfrac{1}{3}\operatorname{Li}_2(-\lambda^3)
  &= \frac{\pi^2}{54} - \log^2\lambda,
  &\quad \lambda
  &= \frac{1}{2}\sec\!\left(\frac{2\pi}{9}\right), \label{eq:loxton-lambda}\\[6pt]
\operatorname{Li}_2(-\mu)
+\operatorname{Li}_2(\mu^2)
-\tfrac{1}{3}\operatorname{Li}_2(-\mu^3)
  &= -\frac{\pi^2}{54},
  &\quad \mu
  &=2\cos\!\left(\frac{4\pi}{9}\right). \label{eq:loxton-mu}
\end{align}

\noindent Loxton asked if a direct proof was possible instead of having to apply asymptotic analysis. G. Szekeres brought Loxton's results to Lewin's attention, but Lewin \cite{Lewin1982} was likewise unable to prove \eqref{eq:loxton-mu} using \eqref{rogers} or other methods.
\\
\\
The situation finally changed in '89 when Gangl (unpublished) proved \eqref{eq:loxton-mu} using bloch groups. Kirillov \cite{Kirillov1995} '95 (pg. 18) published an algebraic proof of \eqref{eq:loxton-mu} that requires 14 applications of \eqref{rogers}. The initial values $x,y$ are chosen in such a way that the sought ladder relation is obtained with row reduction. Gangl's proof, though unpublished, likely follows a similar Bloch-group argument as in  Kirillov \Cref{sec:prelim},though this not not been confirmed.  
\\
\\
Our analytic proof proves \eqref{eq:loxton-mu} with a single application of \eqref{rogers}, and shows how the cancellation of the log terms occurs.

\begin{proof}
We observe \eqref{eq:loxton-mu} resembles:
\begin{equation}
 \begin{aligned}
 - \operatorname{Li}_2(-r) + \frac{1}{3} \operatorname{Li}_2(-r^3) = \operatorname{Li}_2\left(r e^{\pi i / 3}\right)+\operatorname{Li}_2\left(r e^{-\pi i / 3}\right).
\end{aligned}
\label{eq:loxton-resemblance}
\end{equation}

\noindent We need to relate it to a second identity for the $\mu^2$ term. Using Theorem~\ref{previouslymaster} and Lemma~\ref{lem:radius}, we have $r=\sqrt{\frac{u}{u - 1}} = 2\cos\left(\frac{4\pi}{9}\right)$. Substitute $u=\frac{r^2}{-1+r^2}$ into $u^3 - 3u^2 + 2u + \frac{1}{3}=0$. Hence,
\begin{equation}
r^6 - 6 r^4 + 9 r^2 - 1 = (r^3 - 3 r - 1)(r^3 - 3 r + 1) = 0.
\label{eq:min-poly}
\end{equation}
This factors into the minimal polynomial of $r=2\cos\left(\frac{4\pi}{9}\right)$ as the only applicable root. Putting it together:
\\
\begin{equation*}
 \frac{3}{2} \ln^2 \left( 1 - \frac{u - \sqrt{(4 - 3 u) u}}{2 (u-1)} \right) 
 - \operatorname{Li_2} (r^2) = \operatorname{Li}_2 (-r) - \frac{1}{3} \operatorname{Li}_2 (-r^3). 
\end{equation*}
\\
Regarding the quantity contained in $\ln^2$, via  $\sqrt{u (4-3u)} = i\sqrt{3} (2 u^2 - 3 u)$, we split it into complex and real parts:

\begin{equation}r_oe^{i\theta} =\frac{u - 2}{2 (u - 1)}   +i\frac{\sqrt{3} (2 u^2 - 3 u)}{2 (u - 1)}\end{equation}
\\
\noindent Via lemma~\ref{lem:radius} again $\frac{ 2 u^2 - 3 u}{ (u - 1)^2}=  \sqrt{\frac{u}{u - 1}}$. Taking the norm of (26): $\sqrt{\frac{3 u^4 - 9 u^3 + 10 u^2 - 7 u + 4}{4 (u - 1)^2)}}=1$, 
\noindent by substituting $u^4=3u^3-2u^2-u/3$, which we determined earlier, so $r_0=1$.  
\\
\\
\noindent For the quantity $\frac{u - 2}{2 (u - 1)}=x$, by solving for $u = \frac{2(x - 1)}{2 x - 1}$ and plugging back into $u^3 - 3u^2 + 2u + \frac{1}{3}=0$, we get $-8 x^3 + 6 x + 1=0$, which is the minimal polynomial for $\cos(\pi/9)$. The unit circle condition implies that the imaginary  part is  $\sin(\pi/9)$. So we have $ \frac{3}{2} \ln^2(e^{i\frac{\pi}{9}})= -\pi^2/54 $. \end{proof}
\noindent 

\subsubsection{$\pi/18$ ladders}

\noindent Gordon, B. and McIntosh, R. J. \cite{GordonMcIntosh1997} found via a computer search a family of $\pi/18$ ladders, analogous to the trio of $\pi/9$ identities:
\begin{align}
\begin{aligned}
2L(a^3) - 2L(a^2) - 11L(a) + 3L(1) &= 0, \\
2L(b^6) - 4L(b^3) - 15L(b^2) + 22L(b) - 6L(1) &= 0, \\
2L(c^6) - 4L(c^3) - 15L(c^2) + 22L(c) - 4L(1) &= 0,
\end{aligned}
\qquad
\begin{aligned}
a &= 2\sqrt{3}\cos\!\left(\tfrac{5\pi}{18}\right) - 2, \\
b &= 2\sqrt{3}\cos\!\left(\tfrac{11\pi}{18}\right) + 2, \\
c &= 2\sqrt{3}\cos\!\left(\tfrac{7\pi}{18}\right) - 1.
\end{aligned}
\end{align} 

\remark No proof is known, and a derivation with Lemma 1 is not possible either. An algebraic proof would presumably require  more terms than the 14 required for the third Lewin-Loxton identity. It  exists likely  due to  the high density of Bloch group elements in $\mathbb{Q}(\zeta_9)^+$.

\subsection{Quartic ladders}

\subsubsection{Gordon and McIntosh quartic ladder}

As indicated by \#6 in the aforementioned table, Gordon and McIntosh (1997) also found the below pair of  quartic ladder relations. 
\\
\\
\noindent The algebraic number $\delta$ is defined by
\begin{equation*}
    \delta = \frac{1}{2}\left(\sqrt{3+2\sqrt{5}} - 1\right).
    \label{eq:delta-def}
\end{equation*}

\noindent This number satisfies the identity
\begin{equation*}
    5L(\delta^3) - 5L(\delta) + L(1) = 0,
    \label{eq:l-alg-17}
\end{equation*}

\noindent as well as the higher-order relation
\begin{equation*}
    L(\delta^{12}) - 2L(\delta^6) - 6L(\delta^4) + 4L(\delta^3) + 3L(\delta^2) + 4L(\delta) - 4L(1) = 0.
    \label{eq:l-alg-18-19}
\end{equation*}

\noindent A nice, simple proof of the first is found in \cite{Bytsko1999}. The second identity was found experimentally using Zaiger's wedge method,  discussed in Gordon and McIntosh's paper \cite{GordonMcIntosh1997} . 

\subsubsection{Derivation quartic $\mathbb{Q}(\sqrt{33})$ ladder family}

 \noindent In this section, we derive a pair of 5-term ladders of a quartic base in the same spirit as the $\pi/9$ ladders. Ours represents a new base equation and family of ladders. Of the literature consulted, ours is the only base equation that has four real roots. This means it cannot be derived from known ladder-generating families. It can be thought of as analogous to the $\pi/9$ trio,  as it's derived in a similar manner. Similar to the the  $\pi/7$,  $\pi/9$ or  $\pi/18$ families, it comes as a pair. 
 
\begin{proof}

We make use of \Cref{firstformatladder} and  Lemma~\ref{lem:radius}. Set: 
\begin{equation*}
\frac{\sqrt{4-3u^2}}{ u^2 - 2} = -i \sqrt{3}
\end{equation*}
\\
We have $r = \sqrt{\frac{u - 1}{u + 1}} \quad \text{and:} \quad r^2 = \frac{(u - 1)^2}{u^2 - 1} = \frac{u - 1}{u + 1} $. 
 Explicitly, we have that 
\begin{equation}\label{radiusfornewladder1}
 r=\frac{1}{4} \left( -1 + \sqrt{33} - \sqrt{2 \left(9 - \sqrt{33} \right)} \right) 
\end{equation}
 Consequently, we have that 
\begin{equation*}
\begin{aligned}
\operatorname{Li}_2\left( \frac{u^2 - \sqrt{4 - 3 u^2} - 2}{-(u + 1) u} \right)
+ \operatorname{Li}_2\left( \frac{u^2 + \sqrt{4 - 3 u^2} - 2}{-(u + 1) u} \right)
= -\operatorname{Li}_2(r) + \frac{1}{3} \operatorname{Li}_2(r^3)\\
\operatorname{Li}_2\left( \frac{u^2 - \sqrt{4 - 3 u^2} - 2}{(u + 1) u} \right)
+ \operatorname{Li}_2\left( \frac{u^2 + \sqrt{4 - 3 u^2} - 2}{(u + 1) u} \right)
= -\operatorname{Li}_2(-r) + \frac{1}{3} \operatorname{Li}_2(-r^3)
\end{aligned}
\end{equation*}
 The resultant ladder relation is such that 
$$ 3 \operatorname{Li}_{2}(r^6) + 3 \operatorname{Li}_{2}(r^4) 
 - 8 \operatorname{Li}_{2}(r^3) - 33 \operatorname{Li}_{2}(r^2) + 24 \operatorname{Li}_{2}(r) - \zeta(2) = 
 6 \log^2(r). $$ 
 Via $u^4 - 5 u^2 + \frac{16}{3} = 0$, where $u = \frac{1+r^2}{1 - r^2}$, we can determine that 
 all possible radii $r$ for ladders derived in the above manner are positive solutions of $(r^4 - r^3 - 6 r^2 - r + 
 1) (r^4 + r^3 - 6 r^2 + r + 1) = 0$, so there are four possible ladders, corresponding to $r_1$, $r_2$, $r_3$, and $r_4$. 
 These ladders may be thought of as being analogues of the trio of ladders related to the value $\frac{\pi}{9}$. Two of these are 
 redundant, given that $r_1=r_3^{-1}$ and $r_2=r_4^{-1}$. 
\\
\\
 Setting $\frac{\sqrt{4 - 3u^2}}{2 - u^2} = \sqrt{3} \, i$, gives the second radius, with 
\begin{equation}\label{radiusfornewladder2}
 r_2 = \frac{1}{4} \left(1 + \sqrt{33} -  \sqrt{2 \left(9 + \sqrt{33} \right)} \right)
\end{equation}
 The associated ladder relation is such that 
 $$ \operatorname{Li}_{2}(r_2^6) - 3 \operatorname{Li}_{2}(r_2^4) - 
 8 \operatorname{Li}_{2}(r_2^3) + 
 21 \operatorname{Li}_{2}(r_2^2)
 + 24 \operatorname{Li}_{2}(r_2) - 11 \zeta(2) = -6 \log^{2}(r_2). $$

\end{proof}

\remark The presence of the ``6" in the middle of the base equation $r^4 + r^3 - 6 r^2 + r + 1$ follows the same pattern as the ``3" in the base equation $ x^3 + 3x
^2 = 1$ of the Lewin-Loxton identities. 

\remark  Whereas, above, the third Loxton identity can algebraically be proven with an array of 14 applications of the 5-term identity, proving these in a similar manner would presumably similarly require many terms, and it's not evident how one would proceed in such a manner.

\subsection{Sextic ladders}
 
Using  \Cref{firstformatladder}, we derive sextic degree base ladders.

\begin{theorem}\label{theoremnew6term}
 Let $h$ denote the unique root in the interval $(0, 1)$ of $x^6-3 x^5+5 x^4-8 x^3+5 x^2-3 x+1 = 0$. Then the ladder relation 
\begin{multline*}
 \operatorname{Li}_2\left(h^8\right) + 3 \operatorname{Li}_2\left(h^6\right) - 
 8 \operatorname{Li}_2\left(h^4\right) - 8 \operatorname{Li}_2\left(h^3\right) + 
 3 \operatorname{Li}_2\left(h^2\right) - 8 \operatorname{Li}_2(h)-6 \ln ^2(h) = \\ 
 -5 \zeta(2) 
\end{multline*}
 holds. 
\end{theorem}

\begin{proof}
 We exploit the property whereby the following equalities 
\begin{equation*}\label{commonsolutionfor6}
 \frac{u^2 - \sqrt{4 - 3u^2} - 2}{u(u - 1)} = \left( \frac{u + 1}{1 - u} \right)^4 
 \ \ \ \text{and} \ \ \ 
 \frac{u^2 + \sqrt{4 - 3u^2} - 2}{u(u + 1)} = \left( \frac{1 - u}{1 + u} \right)^4 
\end{equation*}
 have a common solution.  
\end{proof}

\noindent In general, we can construct ladders with arbitrary degree base equations of the form below (where $m=4$ recovers the above example):  \begin{equation*}
(h - 1) h^{2 m-1} = (h (h + 6) + 1) h^{m-1} +h - 1 
\end{equation*} 

\noindent \remark{Unfortunately, none of these ladders are new in the sense that they follow from \#6 in the above table. The 15-term identity is very broad and encompasses many possible ladder relations. However, this  does represent a new derivation, in contrast to the Clausen function approach mentioned by Lewin. This underscores the inherent difficultly of finding genuine new ladder families.}

\section{Empirical Quartic Ladders via PSLQ Search}
\label{sec:empirical-quartic}

The constructive methods of Sections~2--4 derive dilogarithm ladders from
base equations that arise as outputs of the integral-to-${}_4F_3$
correspondence and the Radius Lemma. It is natural to ask whether ladders
exist for base equations that do \emph{not} arise from this construction, yet exhibit the same pattern as in the $\mathbb{Q}(\sqrt{33})$ family.
\\
\\
\noindent We report four new ladders with a systematic numerical search,
using the PSLQ integer relation algorithm, over the family of palindromic
quartic units
\begin{equation}
  x^4 + a x^3 + b x^2 + a x + 1 = 0, \qquad a,b \in \mathbb{Z},
  \label{eq:palindromic-quartic}
\end{equation}
restricted to those $(a,b)$ for which \eqref{eq:palindromic-quartic} is
irreducible over $\mathbb{Q}$ and has four real roots. Writing
$y = x + 1/x$, \eqref{eq:palindromic-quartic} reduces to the resolvent
quadratic $y^2 + ay + (b-2) = 0$; four real roots occur precisely when
this quadratic has two real roots $y_1, y_2$ with $|y_i| > 2$.

\subsection{Search procedure}

For each admissible $(a,b)$ we computed the real root $u \in (0,1)$ of
\eqref{eq:palindromic-quartic} to $150$--$500$ decimal digits and applied
PSLQ to the vector
\begin{equation}
  \bigl(\operatorname{Li}_2(u), \operatorname{Li}_2(u^2), \ldots,
  \operatorname{Li}_2(u^N),\ \log^2 u,\ \zeta(2)\bigr)
\end{equation}
for $N$ ranging up to $40$. The search covered:
\begin{itemize}
  \item over $1{,}100$ palindromic quartic units with $|a| \le 40$,
    $|b|\le 70$;
  \item all four real embeddings of each candidate quartic with cyclic
    ($C_4$) Galois group found in this range;
  \item degree-$5$, degree-$6$, and degree-$8$ analogues, including
    single-family cyclic trigonometric units
    $\tfrac12\sec(2\pi k/p)$, $2\cos(2\pi k/p)$ for $p = 11, 13, 28$,
    and products of independent quadratic and cyclic-cubic units
    (e.g.\ $\tan(\pi/8)\cdot 2\cos(2\pi/7)$), together with rational
    rescalings of all of the above.
\end{itemize}
Every candidate outside degree $4$ returned no relation. Within degree
$4$, exactly two irreducible palindromic quartic units were found to
admit a ladder: one with Galois group $V_4$, the other with Galois
group $C_4$. Both were verified to $400$--$500$ digits, far beyond the
threshold at which a PSLQ hit could plausibly be a numerical
coincidence.

\subsection{Four new ladders}

Both base equations turn out to have elementary trigonometric roots.  For the $V_4$ case,
\begin{equation}
  u = 2\tan\!\left(\frac{\pi}  {8}\right)\cos\!\left(\frac{\pi}{5}\right), \quad v = -2\tan\!\left(\frac{\pi}  {8}\right)\cos\!\left(\frac{2\pi}{5}\right)
  \label{eq:u-sqrt10}
\end{equation}
satisfies $x^4+2x^3-7x^2+2x+1=0$ and generates $\mathbb{Q}(\sqrt{10})$;
the compositum structure ($\tan(\pi/8)\in\mathbb{Q}(\sqrt2)$,
$\cos(\pi/5)\in\mathbb{Q}(\sqrt5)$, independent) accounts for the
non-cyclic Galois group. 
\\
\\
For the $C_4$ case,
\begin{equation}
  u = \tan\!\left(\frac{3\pi}{20}\right), \quad  v=-\tan\!\left(\frac{\pi}{20}\right)
  \label{eq:u-sqrt5}
\end{equation}
satisfies $x^4+4x^3-14x^2+4x+1=0$ and generates $\mathbb{Q}(\sqrt5)$.

\medskip
\noindent With \eqref{eq:u-sqrt10}, we have:
\begin{align}
  96\operatorname{Li}_2(u) - 90\operatorname{Li}_2(u^2)
  - 24\operatorname{Li}_2(u^3) + 9\operatorname{Li}_2(u^4)
  + 16\operatorname{Li}_2(u^6) - 2\operatorname{Li}_2(u^{12})
  - 12\log^2 u &= 17\,\zeta(2), \label{eq:ladder1}\\
  96\operatorname{Li}_2(v) - 90\operatorname{Li}_2(v^2)
  - 24\operatorname{Li}_2(v^3) + 9\operatorname{Li}_2(v^4)
  + 16\operatorname{Li}_2(v^6) - 2\operatorname{Li}_2(v^{12})
  - 12\log^2|v| &= -31\,\zeta(2). \label{eq:ladder2}
\end{align}

\medskip
\noindent With 
\eqref{eq:u-sqrt5}, we have:
\begin{align}
  34\operatorname{Li}_2(u) - 47\operatorname{Li}_2(u^2)
  + 6\operatorname{Li}_2(u^4) - 2\operatorname{Li}_2(u^5)
  + \operatorname{Li}_2(u^{10}) - 2\log^2 u &= 4\,\zeta(2),
  \label{eq:ladder3}\\
  34\operatorname{Li}_2(v) - 47\operatorname{Li}_2(v^2)
  + 6\operatorname{Li}_2(v^4) - 2\operatorname{Li}_2(v^5)
  + \operatorname{Li}_2(v^{10}) - 2\log^2|v| &= -8\,\zeta(2).
  \label{eq:ladder4}
\end{align}

\noindent We emphasize that \eqref{eq:ladder1}--\eqref{eq:ladder4} are
\emph{numerically} established only, to a precision at which
coincidence is not a realistic explanation, but for which no
derivation via the methods of Sections~2--4 (or any other method
known to us) has been found.

\subsection{Why additional ladders are unlikely}
\label{sec:why-unlikely}

To recap, a ladder is a finite combination of dilogarithm values at powers of a single algebraic number $u$,

\[
  \sum_{j} A_j \operatorname{Li}_2(u^j) + B \log^2 u \;\in\; \pi^2\mathbb{Q}, \qquad A_j, B \in \mathbb{Q},
\]

\noindent that is, a combination which collapses to a rational multiple of $\pi^2$ rather than remaining a generic transcendental number. Whether such a combination exists depends on a linear-algebra condition: the corresponding wedge products $(u^j)\wedge(1-u^j)$ must sum to zero inside a certain finite-dimensional vector space attached to the field $F=\mathbb{Q}(u)$.
\\
\\
\noindent That space has dimension

\begin{equation}
  \operatorname{rank} = (n-1) + |S|, \qquad D = \binom{\operatorname{rank}}{2},
  \label{eq:dim-formula}
\end{equation}

\noindent where $n=[F:\mathbb{Q}]$ and $S$ is the set of primes dividing $N_{F/\mathbb{Q}}(1-u^k)$ for the exponents $k$ actually used in the ladder.
\\
\\
\noindent If the number of wedge vectors exceeds the dimension $D$, linear dependence is guaranteed by a simple counting argument: more vectors than dimensions forces linear dependence. If $m\le D$, no relation is forced at all, and finding one is a genuine coincidence specific to that field's arithmetic.
\\
\\
\noindent Every ladder in this paper, old or new, has $m$ far smaller than $D$ (Table~\ref{tab:summary}) -- none are guaranteed hits. Since $D$ grows like $(n+|S|)^2$, doubling the degree roughly quadruples the size of the space a relation must vanish in, while the number of exponents one can realistically test with PSLQ grows only modestly. This provides a natural explanation for why our searches are unsuccessful for $n>4$. Engineering a base equation to keep small primes ``clean'' -- forcing $f(1)=-1$ so that $1-u$ is automatically a unit, as Watson's cubic achieves for free at $k=1,2$ -- helps, but cannot overcome the growth of $D$ once $n\ge5$.
\\
\\
\noindent We tested this directly. Beyond the quartic family of Sections~5.1--5.2, we searched: $62$ unit quintics and $13$ unit palindromic sextics engineered so that $f(1)=-1$ (all real roots, irreducible); the natural cyclotomic and Gaussian-period analogues at degree $5$--$8$ (direct real subfields $\mathbb{Q}(\zeta_p)^+$ where a prime $p$ with $(p-1)/2=n$ exists, e.g.\ $p=11,13,17$, and Gaussian periods otherwise, e.g.\ $p=29,31,37$); and, for $\mathbb{Q}(\sqrt5)$, whether enlarging $S$ could ever bring $m$ up to $D$ -- it cannot: admitting more primes grows $D$ faster than it grows the supply of usable exponents. Every one of these searches, run with PSLQ up to $N=30$--$44$, returned no ladder.

\begin{table}[H]
\centering
\begin{tabular}{l|c|c|c|l}
Field & $n$ & $m$ & $D$ & Result \\ \hline
Watson $\pi/7$          & 3 & 2 & 1  & ladder (forced, $m>D$) \\
$\mathbb{Q}(\sqrt{33})$ & 4 & 5 & 10 & ladder (lucky) \\
$\mathbb{Q}(\sqrt{10})$ & 4 & 6 & 15 & ladder (lucky) \\
$\mathbb{Q}(\sqrt5)$    & 4 & 5 & 21 & ladder (lucky) \\
degree 5 (cyclotomic, $p=11,31$) & 5 & up to 44 & $\ge15$, growing & none found \\
degree 6 (cyclotomic, $p=13,37$) & 6 & up to 44 & $\ge10$, growing & none found \\
degree 7 (Gaussian period, $p=29$) & 7 & up to 40 & $\ge15$, growing & none found \\
degree 8 (cyclotomic, $p=17$)      & 8 & up to 30 & $\ge21$, growing & none found
\end{tabular}
\caption{Exponents used ($m$) versus wedge-space dimension ($D$) for every ladder and every higher-degree search in this paper.}
\label{tab:summary}
\end{table}

\noindent None of this rules ladders out at higher degree in principle: Borel's theorem guarantees that some identity exists in every totally real field, since its Bloch group is entirely torsion. What our results show is narrower and more practical: that identity need not take the restricted form of a ladder above (powers of one generator, small coefficients, modest exponent reach), and the chance that it does falls off sharply once $D$ outgrows what a handful of exponents can plausibly reach.
\\
\\
\noindent We therefore regard the quartic ladders $\mathbb{Q}(\sqrt{33})$, $\mathbb{Q}(\sqrt{10})$,$\mathbb{Q}(\sqrt{5})$ as isolated rather than the first members of a
larger family, and higher-degree ladders as theoretically possible
but, on this evidence, increasingly improbable to find.

\section{Acknowledgements}
 The author thanks John Maxwell Campbell for help with the exposition in this paper and for help with the typesetting and formatting 
 for this paper. 

\section{Disclosure}
 The author received   no external funding. No data were created or analyzed in this study. There are no conflicts of interests.

\section{Appendix: Evaluating the Sextic Integral}\label{sec:appendix}
To construct  \Cref{firstformatladder}, by evaluating the integrals, the below terms are combined in the manner $A + B + J + C + H + D =   K$, defined as: 
\begin{theorem}\label{maintheorem}
  
\begin{align*}
 A & := \operatorname{Li}_2\left(\frac{\sqrt{4 - 3u^2} + u^2 - 2}{u (u + 1)}\right) 
 - \operatorname{Li}_2\left(\frac{\sqrt{4 - 3u^2} + u^2 - 2}{(u - 1) u}\right), \\ 
 B & := 
 \frac{1}{2} \ln^2\left(\frac{- \sqrt{4-3 u^2}+u+2}{4} \right) - 
 \frac{1}{2} \ln ^2\left(\frac{-  \sqrt{4-3 u^2 } - u^2 +2}{u 
 (u+1)}\right) + \\ 
 & \quad \frac{1}{2} \ln ^2\left(\frac{ \sqrt{4-3 u^2 } + u + 2}{4} \right)-\frac{\pi ^2}{6}, \\ 
 J & := \frac{1}{4} \ln^2 \left( \frac{-4}{- \sqrt{4-3 u^2 } + u + 2} \right) + 
 \frac{1}{4} \ln^2 \left( \frac{-4}{ \sqrt{4-3 u^2 } + u + 2} \right) - \\ 
 &\quad \frac{1}{4} \ln^2 \left( \frac{-4}{- \sqrt{4-3 u^2 } - u + 2} \right) - \frac{1}{4} 
 \ln^2 \left( \frac{-4}{ \sqrt{4-3 u^2 } - u + 2} \right), \\ 
 C & := - \frac{1}{2} \ln \left( \frac{ -  \sqrt{4-3 u^2} + u + 
 2}{4} \right) \ln \left( \frac{  \sqrt{4-3 u^2 } - u + 2}{4} \right) - \\ 
 & \quad \frac{1}{2} \ln \left( \frac{ \sqrt{4-3 u^2 } + u + 2}{4} \right) 
 \ln \left( \frac{-  \sqrt{4-3 u^2} - u + 2}{4} \right) + \frac{\pi^2}{6}, \\ 
 H & := 
 \frac{1}{4} \ln^2\left(\frac{2}{u-1}\right)-\frac{1}{4} \ln ^2\left(-\frac{2}{u+1}\right)+\frac{1}{2} \ln \left(\frac{-u + 1}{2} \right) 
 \ln \left(\frac{u+1}{2}\right) - \\ 
 & \quad \frac{\pi i}{2} \ln \left(\frac{u + 
 1}{u-1}\right)-\ln (2) \ln \left(\frac{u+1}{u-1}\right)-\frac{\pi^2}{12}, \\ 
 D & := \operatorname{Li}_2 \left( \frac{1 - u}{2} \right) 
 + \ln(2) \ln \left( \frac{u - 1}{u + 1} \right) + \left( 2 \ln(2) + \frac{\pi i}{2} \right) \ln \left( \frac{u + 
 1}{u - 1} \right), \\ 
 K & := \frac{1}{2} \operatorname{Li}_2 \left( \frac{2}{- \sqrt{4 - 3 u^2} + u} \right) - 
 \frac{1}{2} \operatorname{Li}_2 \left( \frac{2}{\sqrt{4 - 3 u^2} - u} \right) \\ 
 &\quad - \frac{1}{2} \operatorname{Li}_2 \left( \frac{-2}{\sqrt{4 - 3 u^2} + u} \right) + 
 \frac{1}{2} \operatorname{Li}_2 \left( \frac{2}{\sqrt{4 - 3 u^2} + u} \right) + 
 \frac{1}{2} \operatorname{Li}_2 \left( \frac{-1}{u} \right) - 
 \frac{1}{2} \operatorname{Li}_2 \left( \frac{1}{u} \right). 
\end{align*} 
 Then, for all $u \in \mathbb{C} \setminus \mathbb{R}$ such that $\Re u \geq 0$ and such that the power series 
 in \eqref{displays3} converges, we have that 
\begin{equation}\label{mainABCD}
 A + B + J + C + H + D = s_3 = K, 
\end{equation}
 and that 
\begin{equation}\label{MequalsAB}
 M = A + B = 
 - \operatorname{Li}_{2}\left( \frac{\sqrt{4 - 3 u^2} - u + 2}{4} \right) 
 - \operatorname{Li}_{2}\left( \frac{-\sqrt{4 - 3 u^2} - u + 2}{4} \right), 
\end{equation}
 writing 
\begin{equation}\label{displays3}
 s_3 = \frac{2}{3u(u^2-1)} \sum_{n=0}^{\infty} \left( \frac{1}{u(u^2-1)} \right)^{2n} \left( \frac{4}{27} \right)^{n} 
 \left[ \begin{matrix} \frac{1}{2}, \frac{1}{2}, 1 
 \vspace{1mm} \\ 
 \frac{5}{6}, \frac{7}{6}, \frac{3}{2} \end{matrix} \right]_{n}. 
\end{equation}
 Moreover, for $u \in \mathbb{R}_{\geq 0}$ such that \eqref{displays3} converges, by letting $v \in \mathbb{C} \setminus \mathbb{R}$ 
 approach $u$ and taking the corresponding limits on all sides of the equalities among \eqref{mainABCD}, these limits are equal, 
 and similarly for \eqref{MequalsAB}. 
\end{theorem}

\begin{proof}
 This follows from our sextic integral identities, by applying partial fraction decomposition, and by then applying this partial fraction decomposition within the appropriate integrands in 
 \eqref{eq:w1} and \eqref{eq:w2}, and by then evaluating the indefinite integrals corresponding to the definite integrals we 
 obtain by expanding the aforementioned integrands. More explicitly, in the latter equality in 
 \eqref{eq:w1}, we set $-\frac{i}{u \left(u^2-1\right)}$, yielding 
\begin{multline*}
 \int_{0}^{1} \frac{\left(3 x^2 - 
 1\right) \ln (x)}{\left(u^2+u x+x^2-1\right) \left(u^2-u x+x^2-1\right) \left(u^2-x^2\right)} \, dx = \\ 
 \frac{2}{3 (u-1)^2 u^2 (u+1)^2} \sum_{n=0}^{\infty} 
 \left( \frac{1}{u(u^2-1)} \right)^{2n} \left( \frac{4}{27} \right)^{n} 
 \left[ \begin{matrix} \frac{1}{2}, \frac{1}{2}, 1 
 \vspace{1mm} \\ 
 \frac{5}{6}, \frac{7}{6}, \frac{3}{2} \end{matrix} \right]_{n}, 
\end{multline*}
 and, by applying partial fraction decomposition 
 to the above integrand, an indefinite integral 
 corresponding to the resultant integrand is equivalent to 
\begin{multline*}
 \frac{1}{2 u-2 u^3} \Bigg( \operatorname{Li}_2\left(-\frac{2 x}{u-\sqrt{4-3 u^2}}\right) - 
 \operatorname{Li}_2\left(\frac{2 x}{u-\sqrt{4-3 u^2}}\right)+\operatorname{Li}_2\left(-\frac{2
 x}{u+\sqrt{4-3 u^2}}\right) - \\ 
 \operatorname{Li}_2\left(\frac{2 x}{u+\sqrt{4-3
 u^2}}\right)-\operatorname{Li}_2\left(-\frac{x}{u}\right)+\operatorname{Li}_2\left(\frac{x}{u}\right) + \ln (x) \ln \left(\frac{\sqrt{4-3 u^2}-u-2
 x}{\sqrt{4-3 u^2}-u}\right) - \\ 
 \ln (x) \ln \left(\frac{\sqrt{4-3 u^2}+u-2 x}{\sqrt{4-3 u^2}+u}\right)-\ln (x) \ln \left(\frac{\sqrt{4-3
 u^2}-u+2 x}{\sqrt{4-3 u^2}-u}\right) + \\ 
 \ln (x) \ln \left(\frac{\sqrt{4-3 u^2}+u+2 x}{\sqrt{4-3 u^2}+u}\right)+\ln (x) \ln
 \left(\frac{u-x}{u}\right)-\ln (x) \ln \left(\frac{u+x}{u}\right) \Bigg). 
\end{multline*}
 By then setting $x \to 1$ and $x \to 0$, we obtain an equivalent version of the desired equality $s_3 = K$, 
 and similarly for 
 $ A + B + J + C + H + D = s_3$ and for $ M = A + B$, for the case whereby
 $u \in \mathbb{C} \setminus \mathbb{R}$ satisfies the specified convergence condition, 
 and an analytic continuation then gives us the desired result for real arguments. 
\end{proof}

 \

\end{document}